\documentclass[11pt,twoside,a4paper]{article}
\usepackage{amsmath, amsthm, amssymb, enumerate}
\usepackage{graphicx}
\graphicspath{ {Images/} }

\begin{document}

\title{Multicolour Ramsey Numbers of Odd Cycles}
\date{}
\author{A. Nicholas Day$^{*}$ and J. Robert Johnson\thanks{School of Mathematical Sciences, Queen Mary University of London, London E1 4NS, UK} }
\date{\today}
\maketitle 
\newtheorem{Thm}{Theorem}
\newtheorem{Lemma}[Thm]{Lemma}
\newtheorem{Def}[Thm]{Definition}
\newtheorem{Coro}[Thm]{Corollary}
\newtheorem{Conj}[Thm]{Conjecture}
\newtheorem{Prop}[Thm]{Proposition}
\newtheorem{Quest}[Thm]{Question}

\begin{abstract}
We show that for any positive integer $r$ there exists an integer $k$ and a $k$-colouring of the edges of $K_{2^{k}+1}$ with no monochromatic odd cycle of length less than $r$. This makes progress on a problem of Erd\H{o}s and Graham and answers a question of Chung.  We use these colourings to give new lower bounds on the $k$-colour Ramsey number of the odd cycle and prove that, for all odd $r$ and all $k$ sufficiently large, there exists a constant $\epsilon = \epsilon(r) > 0$ such that $R_{k}(C_{r}) > (r-1)(2+\epsilon)^{k-1}$.

\end{abstract}

\section{Introduction}

In this paper all \textit{colourings} of a graph $G$ will refer to colourings of the edges of $G$.  The \textit{odd girth} of $G$, written $\text{og}(G)$, is the length of the shortest odd cycle in $G$.  Given a colouring $\mathcal{C}$ of $G$ we say the odd girth of $\mathcal{C}$, written $\text{og}(\mathcal{C})$, is the length of the shortest monochromatic odd cycle found in $\mathcal{C}$.    It is a simple exercise to see that it is possible to $k$-colour the complete graph $K_{2^{k}}$ such that each colour comprises a bipartite graph.  Moreover, such colourings only exist for $K_{n}$ if $ n \leqslant 2^{k}$.  Indeed, consider labelling each vertex of $K_{n}$ with a binary vector of length $k$, where the $i^{th}$ coordinate of the label given to a vertex is determined by which side of the bipartition of colour $i$ the vertex lies in.  All vertices of $K_{n}$ must receive distinct labels, and so $n \leqslant 2^{k}$.  It follows that any $k$-colouring of $K_{2^{k}+1}$ must contain a monochromatic odd cycle.  Based on this observation, Erd\H{o}s and Graham \cite{ErdosGraham} asked the following question:

\begin{Quest}\label{Quest1}
How large can the smallest monochromatic odd cycle in a $k$-colouring of $K_{2^{k}+1}$ be?
\end{Quest}

Moreover, Chung \cite{Chung} asked further whether or not this quantity is unbounded as $k$ increases.  In this paper we show that the size of the least odd cycle that must appear is indeed unbounded.

\begin{Thm}\label{Thm1}
For all positive integers $r$ there exists an integer $k$ and a $k$-colouring of  $K_{2^{k}+1}$ with odd girth at least $r$.
\end{Thm}

The proof of Theorem \ref{Thm1} can be found in Section \ref{Section2}. From a quantitative perspective, our proof of Theorem \ref{Thm1} will show that there exist $k$-colourings of $K_{2^{k}+1}$ with odd girth at least $2^{\sqrt{2 \log_{2}(k) - c}}$ for some constant $c$.  This result is a consequence of Corollary \ref{Coro1} which can be found at the end of Section \ref{Section2}.

For a graph $H$, the $k$-colour Ramsey number $R_{k}(H)$ is defined as the least integer $n$ such that every $k$-colouring of $K_{n}$ contains a monochromatic copy of $H$.  We say that a colouring of a graph $G$ is $H$\textit{-free} if it contains no monochromatic copy of $H$.  Erd\H{o}s and Graham \cite{ErdosGraham} showed that 

\begin{equation}\label{eq1}
R_{k}(C_{r}) \geqslant (r-1)2^{k-1} +1
\end{equation}
whenever $r \geqslant 3$ is an odd integer.  The construction used to show this is as follows: When $k = 1$ simply take a $1$-colouring of $K_{r-1}$, for $k>1$ take two disjoint copies of the construction for $k-1$ and colour every edge between the two copies with a new colour.  This construction led Bondy and Erd\H{o}s \cite{BondyErdos} to make the following conjecture.
\begin{Conj} \label{Conj1}
$\textup{(Bondy, Erd\H{o}s)}$ Equality holds in (\ref{eq1}) for all odd integers $r > 3$.
\end{Conj}
In this paper we disprove Conjecture \ref{Conj1} by using the result of Theorem \ref{Thm1} to construct colourings that give new lower bounds for $R_{k}(C_{r})$ whenever $r$ is an odd integer and $k$ is sufficiently large.

\begin{Thm}\label{Thm2}
For all odd integers $r$ there exists a constant $\epsilon = \epsilon(r) > 0$ such that, for all $k$ sufficiently large, $R_{k}(C_{r}) > (r-1)(2+\epsilon)^{k-1}$.
\end{Thm}

The proof of Theorem \ref{Thm2} can be found in Section \ref{Section3}.  We remark that Theorem \ref{Thm2} can not be used to say anything about the behaviour of $R_{k}(C_{r})$ when $k$ is fixed and $r$ is increasing.  Bondy and Erd\H{o}s \cite{BondyErdos} showed that their conjecture holds for all $r$ when $k = 2$.  For $k = 3$, \L uczak \cite{Luczak} employed the regularity method to prove that Bondy and Erd\H{o}s's conjecture holds asymptotically, showing that $R_{3}(C_{r}) = 4r + o(r)$ for odd $r$ increasing.  Kohayakawa, Simonovits and Skokan \cite{KohSimSko} used \L uczak's  method together with stability methods to show that the conjecture is true for $k = 3$ when $r$ is sufficiently large.  Recently, Jenssen and Skokan \cite{JenssenSkokan} showed that Conjecture \ref{Conj1} is true for all fixed $k$ and all $r$ sufficiently large.  They achieved this by using \L uczak's regularity method to turn the problem into one in convex optimisation.

When $r = 3$, it is well known that equality does not hold in (\ref{eq1}).  A result of Fredricksen and Sweet \cite{FredSweet} on \textit{Sum-Free Partitions} shows that $R_{k}(C_{3}) \geqslant c(3.1996 \ldots)^{k}$ for some constant $c$.  We refer the reader to Abbott and Hanson's paper \cite{AbbottHanson} for details about the connection between sum-free partitions and Ramsey Numbers.  As an upper bound, Greenwood and Gleason \cite{GleasonGreenwood} showed that  $R_{k}(C_{3}) \leqslant ek! + 1$, see also Schur \cite{Schur}.  It is a famous open problem to determine whether or not $R_{k}(C_{3})$ is super-exponential in $k$.
\vspace{-1.5 mm}
\section{Colourings with no Short Odd Cycles}\label{Section2}
\vspace{-0.5mm}
For a set $X$, let $X^{(2)} = \big\{ \{ x,y \} : x,y \in X, x \neq y \big\}$.  Given a colouring $\mathcal{C}$ of a graph $G$, let $G(\mathcal{C}_{i})$ be the graph on vertex set $V(G)$ whose edges are those of $G$ that received colour $i$ in $\mathcal{C}$.  Given an edge $ \{x,y\} \in E(G)$, let $\mathcal{C}(x,y)$ be the colour $\{x,y\}$ receives in $\mathcal{C}$.  We begin by focusing our attention on the odd girth of each $G(\mathcal{C}_{i})$ rather than the odd girth of $\mathcal{C}$ as a whole.  We say that $\mathcal{C}$ is an $(r_{1},\ldots,r_{k})$-colouring of $G$ if  $ \text{og}\big( G(\mathcal{C}_{i}) \big) \geqslant r_{i}$ for each $i$.  The first main idea of the proof of Theorem \ref{Thm1} is that we would like to show that if there exists an $(r_{1},\ldots,r_{k})$-colouring of $K_{2^{k}+1}$ then we can use it to build an $(r_{1}+2,r_{2},\ldots,r_{k},r_{k+1})$-colouring of $K_{2^{k+1}+1}$, where $r_{k+1} \geqslant r_{1}+2$.  Given this, we would apply this idea inductively, relabelling the colours at each step so that $r_{1}$ is minimal, to find $k$-colourings of $K_{2^{k}+1}$ (for some $k$) with arbitrarily high odd girth.

Unfortunately we are unable to come up with such a construction for general $(r_{1},\ldots,r_{k})$-colourings.  As a result, the second main idea of our proof will be to impose stronger conditions on our colourings that will allow an induction argument to hold.  We say a graph $G$ is $r$\textit{-round} if there exists a partition of $V(G)$ into sets $(X_{1},\ldots,X_{r})$ such that each edge of $G$ lies between one of the pairs $ (X_{1}, X_{2}),(X_{2}, X_{3}),\ldots,(X_{r-1}, X_{r})$ or $(X_{1}, X_{r})$.  When $r$ is an odd integer we have that any odd cycle in an $r$-round graph $G$ must contain at least one edge between each such pair and so $\text{og}(G) \geqslant r$.

We say that an $r$-round graph $G$ is \textit{rooted} with root $O$, for some vertex $O \in V(G)$, if $X_{1} = \{ O \}$.  We say a $k$-colouring of $G$ is an $(r_{1},\ldots,r_{k})$\textit{-rooted-round-colouring}, and write $(r_{1},\ldots,r_{k})$-RRC, if there exists a vertex $O \in V(G)$ such that $G(\mathcal{C}_{i})$ is a rooted $r_{i}$-round graph with root $O$  for each $i$.  We call $O$ the root of the colouring. 

Note that all $(r_{1},\ldots,r_{k})$-RRC's are $(r_{1},\ldots,r_{k})$-colourings. More generally, it is straightforward yet slightly tedious to prove that a colouring $\mathcal{C}$ of a graph $G$ is an $(r_{1},\ldots,r_{k})$-RRC with root $O \in V(G)$ if and only if $og( G ( \mathcal{C}_{i} ) ) \geqslant r_{i}$ for each $i$ and all monochromatic odd cycles of the RRC go through $O$.  We do not make use of this fact in our argument and so omit its proof from this paper.

The main tool for the proof of Theorem \ref{Thm1} is the following lemma.
\begin{Lemma}\label{Lemma1}

Let $r_{1}, \ldots, r_{k}$ and $n$ be positive integers.  If there exists an $(r_{1},\ldots,r_{k})$-RRC of $K_{n}$ then there exists an $(r_{1}+2,r_{2}, \ldots , r_{k},2r_{1} -1)$-RRC of $K_{2n-1}$.

\end{Lemma}
\begin{proof}[Proof of Lemma \ref{Lemma1}]

Let $G$ be the complete graph on $n$ vertices and let $\mathcal{A}$ be an $(r_{1},\ldots,r_{k})$-RRC of $G$ with root $O$.  For each $i$, let $(O  , X^{i}_{2}, \ldots,X^{i}_{r_{i}} )$ be the partition of $V(G)$ that realises $G(\mathcal{A} _{i})$ as an $r_{i}$-round graph.

Let $H$ be the complete graph on $2n-1$ vertices with vertex set $\{ O \} \cup U \cup U'$ where $U = V(G) \setminus \{O \}$ and $U' = \{ x' : x \in U \}$ is a copy of $U$.  For each pair of integers $i,j$, with $j\geqslant 2$, we define $Y^{i}_{j} \subseteq U'$ to be the set $\{ x' : x \in X^{i}_{j} \}$.  Let $\mathcal{B}$ be the following $(k+1)$-colouring of $H$:
\begin{enumerate}

\item $\mathcal{B}(O,x) = \mathcal{B}(O,x') = \mathcal{A}(O,x)$ for all $x \in U$,
\item $\mathcal{B}(x,y) = \mathcal{B}(x',y) = \mathcal{B}(x,y') = \mathcal{B}(x',y') = \mathcal{A}(x,y)$ for all $\{x,y\}  \in U^{(2)}$,
\item $\mathcal{B}(x,x') = k+1$ for all $x \in U$.

\end{enumerate}

It is easy to check that every edge of $H$ is coloured by $\mathcal{B}$.  We now modify $\mathcal{B}$ to obtain a new colouring, which we call $\mathcal{C}$, that will be our desired $(r_{1}+2,r_{2}, \ldots , r_{k},2r_{1} -1)$-RRC of $H$.  Let $F$ be the following set of edges in $H$:

\begin{enumerate}
\item all edges that lie between $O$ and $X^{1}_{2}$,
\item all edges that lie between $Y^{1}_{l}$ and $X^{1}_{l+1}$ for each $l = 2,\ldots,r_{1}-1$,
\item all edges between $O$ and $Y^{1}_{r_{1}}$.
\end{enumerate}

Moreover, let $F(\mathcal{B}_{1})$ be the set of edges in $F$ that have colour $1$ in $\mathcal{B}$.   We obtain $\mathcal{C}$ from  $\mathcal{B}$ by giving colour $k+1$ to all edges in $F(\mathcal{B}_{1})$.  All other edges of $H$ receive the same colour in $\mathcal{C}$ as they did in $\mathcal{B}$. See Figure \ref{Fig1} for an illustration of the colours $1$ and $k+1$ in $\mathcal{C}$.  To complete the proof of Lemma \ref{Lemma1} we note that is easy to verify the following three statements:

\begin{enumerate}
\item $H(\mathcal{C}_{1})$ is a rooted $(r_{1} + 2)$-round graph with root $O$ and partition $\big( O,Y^{1}_{2},Y^{1}_{3}, W_{4}, W_{5}, W_{6}, \ldots, W_{r_{1}},X^{1}_{r_{1}-1},X^{1}_{r_{1}} \big)$ where $W_{l} = X^{1}_{l-2} \cup Y^{1}_{l} $ for each $l\geqslant 4$.
\item  $H(\mathcal{C}_{i})$, for $i = 2,3,\ldots,k$, is a rooted $r_{i}$-round graph with root $O$ and partition $\big(O , X^{i}_{2} \cup Y^{i}_{2}, X^{i}_{3} \cup Y^{i}_{3}, \ldots,X^{i}_{r_{i}} \cup Y^{i}_{r_{i}}\big)$.
\item $H(\mathcal{C}_{k+1})$ is a rooted $(2r_{1} -1)$-round graph with root $O$ and partition $\big( O, X^{1}_{2},Y^{1}_{2}, X^{1}_{3},Y^{1}_{3}, \ldots,X^{1}_{r_{1}},Y^{1}_{r_{1}} \big)$.
\end{enumerate}
\end{proof}
\vspace{-1.95mm}
\begin{figure}[ht]
    \centering
	\includegraphics[scale=0.8]{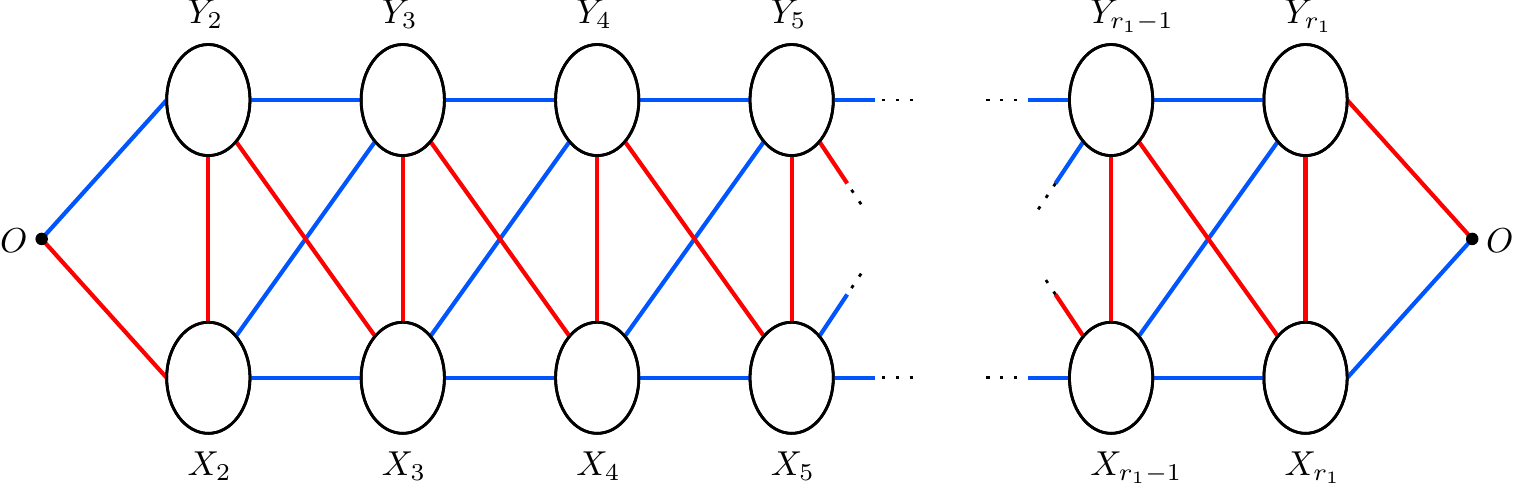}
    \caption{This diagram shows the colours $1$ and $k+1$ in the colouring $\mathcal{C}$ of $H$.  The blue lines represent edges of colour $1$ while the red lines represent edges of colour $k+1$.  For visual clarity we have suppressed the superscripts from each of the sets $X_{i}^{1}$ and $Y_{i}^{1}$, and also drawn the vertex $O$ twice, once at each end of the picture.}
	\label{Fig1}
\end{figure}

The proof of Theorem \ref{Thm1} follows almost immediately from Lemma \ref{Lemma1}.
\begin{proof}[Proof of Theorem \ref{Thm1}]

Note that if $n = 2^{k} + 1$ then $2n-1 = 2^{k+1} + 1$. Consider $2$-colouring the complete graph on vertex set $\{1,2,3,4,5\}$ by colouring the edges $\{(1,2),(2,3),(3,4),(4,5),(1,5) \} $ red and colouring the remaining edges blue.  This is a $(5,5)$-RRC where any vertex can be chosen as the root.  Starting from this $(5,5)$-RRC of $K_{5}$, we can inductively apply Lemma \ref{Lemma1}, relabelling the colours so that $r_{1}$ is minimal at each step, to find $k$-colourings of $K_{2^{k}+1}$ with arbitrarily high odd girth. 

\end{proof}

As noted in the introduction, Erd\H{o}s and Graham \cite{ErdosGraham} were interested in how large the odd girth of a $k$-colouring of $K_{2^{k}+1}$ can be.  It can be easily observed from the proofs of Theorem \ref{Thm1} and Lemma \ref{Lemma1} that if $k \geqslant \sqrt{2}^{(r-3)}$ (and $r$ is odd) then there exist $k$-colourings of $K_{2^{k}+1}$ with odd girth at least $r$.  Indeed, given an RRC of a complete graph,  we can apply Lemma \ref{Lemma1} once to each of its colours to obtain a new RRC with twice as many colours and odd girth at least $2$ larger.  This shows that for any $k$ there exists a $k$-colouring of $K_{2^{k}+1}$ with odd girth at least $2 \log_{2}(k) + 1$.  This bound is simple to obtain yet it is far from the exact behaviour of our sequence of colourings.  In Corollary \ref{Coro1} we analyse this sequence more carefully to obtain improved bounds.

\begin{Coro}\label{Coro1}
There exists a constant $c \approx 4.7685$ such that if $k \geqslant c \sqrt {2}^{(t^{2} - 3t + 2)}$ then there exists a $k$-colouring of $K_{2^{k}+1}$ with odd girth greater than $2^{t}$.
\end{Coro}

\begin{proof}[Proof of Corollary \ref{Coro1}]
Given a sequence of integers $(r_{1},\ldots,r_{k})$, let 
\vspace*{-1mm}
\begin{equation}
f \big( (r_{1},\ldots,r_{k}) \big) = (r_{1}+2,r_{2}, \ldots , r_{k},2r_{1} - 1). \nonumber
\end{equation}

Let $\mathbf{r}_{2} = (5,5)$ and for $j>2$ let $\mathbf{r}_{j}$ be the sequence obtained from rearranging the entries of $f(\mathbf{r}_{j-1})$ in increasing order.  The definition of the function $f$ comes from the construction given in Lemma \ref{Lemma1} and the starting sequence $\mathbf{r}_{2}$ corresponds to our $(5,5)$-RCC of $K_{5}$.

To prove our corollary it is sufficient to analyse how quickly the minimum value of $\mathbf{r}_{j}$ grows as $j$ increases. Define the function 
\begin{equation}
p(t) = \prod_{i = 0}^{t-2} \Big(2^{i} + 1 \Big). \nonumber
\end{equation}
We show by induction on $t$ that $\min(\mathbf{r}_{p(t)})\geqslant 2^{t} + 1$.  Our base case holds when $t = 2$ as $\min(\mathbf{r}_{2}) = 5$.  Suppose the statement holds true for $t-1$.  As $\min(\mathbf{r}_{p(t-1)})\geqslant 2^{t-1} + 1$, we have that each time we apply $f$ to $\mathbf{r}_{p(t-1)}$ (and rearrange its elements in increasing order), the newest element we've added is at least $2^{t} + 1$.  Thus, to find $m$ such that $\min(\mathbf{r}_{m})\geqslant 2^{t} + 1$, we are only concerned with the ``adding $2$" process of $f$.  As $\mathbf{r}_{p(t-1)}$ has $p(t-1)$ elements, each at least $2^{t-1} +1 $, we only need to repeat the process of adding $2$ to its minimal element at most $2^{t-2}p(t-1)$ times to find a sequence whose minimum value is at least $2^{t}+1 $.  Thus, as $2^{t-2}p(t-1) + p(t-1) = p(t)$, our inductive statement holds true for $t$.  To complete the proof of the corollary we note that
\begin{eqnarray}
p(t) &=& \sqrt{2}^{ (t^{2} -3t +2)}\prod ^{t-2}_{i=0} \Big( 1+\frac{1}{2^{i}} \Big) \nonumber \\
& < & c \sqrt{2}^{ (t^{2} -3t +2)}, \nonumber
\end{eqnarray}
where $c = \prod \limits_{i \geqslant 0} ( 1+\frac{1}{2^{i}} ) \approx 4.7685$.
\end{proof}

Corollary \ref{Coro1} shows that there exist $k$-colourings of $K_{2^{k}+1}$ with odd girth at least $2^{\sqrt{2 \log_{2}(k) - c_{0}}}$ for some constant $c_{0}$.  We note that Corollary \ref{Coro1} still does not give the  exact bound for the behaviour of our sequence of colourings.  Moreover, we do not believe that our RRCs give rise to the best possible bounds that one could hope for from general colourings.  Indeed, the colourings we construct have some colours with odd girth almost twice as large as some other colours.  It seems more likely that the colourings with the largest odd girth will be more balanced across all of their colours.  As such, we have not analysed the minimum values of the sequence of $\mathbf{r}_{j}$'s any more carefully to obtain better bounds.
\vspace{-2mm}

\section{Multicolour Ramsey Numbers of Odd Cycles}\label{Section3}
\vspace{-1mm}
In this section we often write the pair $(G,\mathcal{C} )$ to refer to a colouring $\mathcal{C}$ of a graph $G$.  This will allow us to simultaneously keep track of multiple colourings of different graphs.  Let $G$ and $H$ be complete graphs on $m$ and $n$ vertices respectively and let $G \times H$ be the complete graph with vertex set $V(G) \times V(H)$.  Moreover, let $(G,\mathcal{A})$ be a $j$-colouring  using colours $1,\ldots,j$ and let $(H,\mathcal{B})$ be a $k$-colouring using colours $j+1,\ldots,j+k$.  We define the set of colourings $\times (\mathcal{A},\mathcal{B})$ to be all $(k+j)$-colourings $\mathcal{C}$ of $ G \times H$  such that

\begin{enumerate}
\item $\mathcal{C} \big( (g_{1},h),(g_{2},h) \big) = \mathcal{A}(g_{1},g_{2})$ for all $ \{ g_{1},g_{2} \} \in V(G)^{(2)}$, $h \in V(H)$,
\item $\mathcal{C} \big( (g,h_{1}),(g,h_{2}) \big) = \mathcal{B}(h_{1},h_{2})$ for all $g \in V(G)$, $\{ h_{1},h_{2} \} \in V(H)^{(2)}$,
\item $\mathcal{C} \big( (g_{1},h_{1}),(g_{2},h_{2}) \big) = \mathcal{A}(g_{1},g_{2})$ or $\mathcal{B}(h_{1},h_{2})$ for all $\{ g_{1},g_{2} \} \in V(G)^{(2)}$, $\{ h_{1},h_{2} \} \in V(H)^{(2)}$.
\end{enumerate}

Note that $\times (\mathcal{A},\mathcal{B})$ is a set of colourings as we have a choice of colour for many of the edges of $G \times H$.  We call $\times (\mathcal{A},\mathcal{B})$ the set of \textit{product colourings} of $(G,\mathcal{A})$ and $(H,\mathcal{B})$.  In particular, we define $(G \times H,\mathcal{A} * \mathcal{B})$ to be the product colouring with $\mathcal{C} \big( (g_{1},h_{1}),(g_{2},h_{2}) \big) = \mathcal{B}(h_{1},h_{2})$ for all $\{ g_{1},g_{2} \} \in V(G)^{(2)}$, $\{ h_{1},h_{2} \} \in V(H)^{(2)}$.  $(G \times H,\mathcal{A} * \mathcal{B})$ can be thought of as the colouring that arises from replacing every vertex of $H$ in $(H,\mathcal{B})$ with a copy of $(G,\mathcal{A})$.

Given an odd integer $r$, we would like to use product colourings to build new $C_{r}$-free colourings from other known $C_{r}$-free colourings.  Unfortunately, if $(G,\mathcal{A})$ and $(H,\mathcal{B})$ are both $C_{r}$-free then it is not necessarily the case that there are any colourings in $\times (\mathcal{A},\mathcal{B})$ that are also $C_{r}$-free.  Nevertheless, product colourings do have the useful property of ``preserving odd girth" and allow us to build new $C_{r}$-free colourings, subject to the right conditions as described by the first part of Lemma \ref{Lemma2}.  This, with Theorem \ref{Thm1}, is already enough to construct colourings that disprove Bondy and Erd\H{o}s's conjecture (Conjecture \ref{Conj1}).  However, in order to prove Theorem \ref{Thm2} as stated, we use the colouring $(G \times H,\mathcal{A} * \mathcal{B})$.  The second part of Lemma \ref{Lemma2} describes how this colouring preserves odd girth in an even stronger sense than general product colourings.

\begin{Lemma}\label{Lemma2}

\  \begin{enumerate}
\item Let $r$ be an integer and suppose that $(G,\mathcal{A})$ and $(H,\mathcal{B})$ are colourings with odd girth at least $r$.  Then any colouring of $G \times H$ in $\times(\mathcal{A},\mathcal{B})$ also has odd girth at least $r$.

\item Let $r$ be an odd integer.  Suppose $(G,\mathcal{A})$ is a $C_{r}$-free colouring and $(H,\mathcal{B})$ is a colouring with odd girth strictly greater than $r$.  Then the colouring $(G \times H,\mathcal{A} * \mathcal{B})$ is $C_{r}$-free.
\end{enumerate}
\end{Lemma}

\begin{proof}[Proof of Lemma \ref{Lemma2}]

We first prove part 1 of the lemma.  Let $(G \times H, \mathcal{C}) \in \times(\mathcal{A},\mathcal{B})$ and suppose that $(g_{1},h_{1}),(g_{2},h_{2}),\ldots,(g_{r'},h_{r'})$ is a monochromatic odd cycle in $\mathcal{C}$ with $r' < r$.  Without loss of generality we may assume that the colour of this monochromatic cycle is one of the colours that appears in $(H,\mathcal{B})$.  Under this assumption, we have that $h_{1},h_{2},\ldots,h_{r'}$ is a closed monochromatic odd walk in $(H,\mathcal{B})$.  It is easy to see that any such closed odd walk must contain an odd cycle of length less than or equal to $r'$. This contradicts og$(\mathcal{B}) \geqslant r > r'$. 

We now prove part 2 of the lemma.  Suppose $(g_{1},h_{1}),(g_{2},h_{2}),\ldots,(g_{r},h_{r})$ is a monochromatic cycle in $(G \times H,\mathcal{A} * \mathcal{B})$.  As og$(\mathcal{B}) > r$, an identical argument to the above proof of part 1 shows that this cycle cannot be in one of the colours in $(H,\mathcal{B})$.  Thus we may assume that the cycle is in one of the colours that appears in $(G,\mathcal{A})$.  However, by the definition of $\mathcal{A} * \mathcal{B}$, we would have that $g_{1},g_{2},\ldots,g_{r}$ is a monochromatic cycle of length $r$ in $(G,\mathcal{A})$, contradicting our assumption that $(G,\mathcal{A})$ is $C_{r}$-free.

\end{proof}

\begin{proof}[Proof of Theorem \ref{Thm2}]

Let $r$ be an odd integer.  By Theorem \ref{Thm1}, there exists a least integer $f = f(r)$ and an $f$-colouring of $K_{2^{f} + 1}$, which we call $\mathcal{A}$,  with odd girth strictly greater than $r$.  Given an integer $k$, let $m$ and $c$ be positive integers such that $k-1 = mf +c$ where $f>c \geqslant 0$.  As noted in the introduction, Erd\H{o}s and Graham \cite{ErdosGraham} showed that there exists a $C_{r}$-free $(c+1)$-colouring of $K_{n}$, where $n = (r-1)2^{c}$.  Call this $(c+1)$-colouring $\mathcal{B}_{0}$ and for $i \geqslant 1$ let $\mathcal{B}_{i} = \mathcal{B}_{i-1}*\mathcal{A}$.  By Lemma \ref{Lemma2}, the colouring $\mathcal{B}_{m}$ is a $k$-colouring of the complete graph on $(r-1)2^{c}(2^{f}+1)^{m}$ vertices with no monochromatic cycle of length $r$.  For $\epsilon > 0$ sufficiently small and $k$ (equivalently $m$) sufficiently large, this graph has more than $(r-1)(2 + \epsilon )^{k-1}$ vertices.

\end{proof}

\section{Acknowledgements}

The first author was supported by an EPSRC doctoral studentship.

\end{document}